\documentclass[reqno, 12pt]{amsart}

\usepackage[utf8]{inputenc}

 \usepackage{tikz-cd}

\usepackage{amsfonts, amsthm, amsmath, amssymb}
\usepackage{hyperref}
\hypersetup{colorlinks=false}
\usepackage[all]{xy}
 
 \usepackage{pifont}

 \usepackage[margin=1in]{geometry}

\usepackage[french,bordercolor=black,backgroundcolor=red,linecolor=red,textsize=footnotesize,textwidth=0.75in,shadow]{todonotes}

\usepackage{hyperref}
\usepackage{enumerate}

\numberwithin{equation}{section}

\newtheorem{theorem}{Theorem}[section] 
\newtheorem*{theorem*}{Theorem}
\newtheorem{lemma}[theorem]{Lemma}

\newtheorem{proposition}[theorem]{Proposition}

\newtheorem{corollary}[theorem]{Corollary}

\theoremstyle{definition}
\newtheorem*{acknowledgements}{Acknowledgements}
\newtheorem{remark}[theorem]{Remark}
\newtheorem{definition}[theorem]{Definition}
\newtheorem{example}[theorem]{Example}

\renewcommand{\phi}{\varphi}

\renewcommand{\leq}{\leqslant}

\renewcommand{\geq}{\geqslant}

\renewcommand{\bar}{\overline}

\newcommand{\x}{\mathbf{x}}

\renewcommand{\t}{\mathbf{t}}

\renewcommand{\a}{\mathbf{a}}

 \DeclareMathOperator{\disc}{disc}

\DeclareMathOperator{\rad}{rad}
\DeclareMathOperator{\Pic}{Pic}

\DeclareMathOperator{\Spec}{Spec}

\DeclareMathOperator{\Jac}{Jac}

\DeclareSymbolFont{bbold}{U}{bbold}{m}{n}
\DeclareSymbolFontAlphabet{\mathbbold}{bbold}

\renewcommand{\P}{\mathbb{P}}
\newcommand{\Q}{\mathbb{Q}} 

\newcommand{\F}{\mathbb{F}}
\newcommand{\N}{\mathbb{N}}
\newcommand{\R}{\mathbb{R}}

\newcommand{\Z}{\mathbb{Z}}
\newcommand{\A}{\mathbb{A}}

\renewcommand{\epsilon}{\varepsilon}

\renewcommand{\leq}{\leqslant}
\renewcommand{\geq}{\geqslant}

\title{Normal distribution of bad reduction}

\usepackage{color}

\author{Robert J. Lemke Oliver}
\address{ 
Department of Mathematics \\
Tufts University \\
Medford, MA \\
02155 \\
USA}
\email{robert.lemke\_{}oliver@tufts.edu}

\author{Daniel Loughran} 
  \address{
  Department of Mathematical Sciences \\
University of Bath \\
Claverton Down \\
Bath \\
BA2 7AY \\
UK}
\urladdr{https://sites.google.com/site/danielloughran}

\author{Ari Shnidman} 
\address{
Einstein Institute of Mathematics\\
Hebrew University of Jerusalem\\
Israel
}
\email{ariel.shnidman@mail.huji.ac.il}

 \subjclass[2010]
{11G30; 
60F05, 
}
\date{\today}

\begin{document}

\begin{abstract}
	We prove normal distribution laws for primes of bad semistable reduction in families of curves. As a consequence, we deduce that when ordered by height, $100\%$ of curves in these families  have, in a precise sense, many such primes.
\end{abstract}

\maketitle

\setcounter{tocdepth}{1}
\tableofcontents

\vspace{-1,5cm}

\section{Introduction}
A famous theorem of Erd\H{o}s and Kac \cite{EK40} states that the function $\omega(n) = \#\{\mbox{primes } p \colon p \mid n\}$ behaves like a normal distribution with mean and variance $\log \log n$; more precisely the random variables
$$\{ n \in \N : n \leq B\}  \to \R, \quad n \mapsto \frac{\omega(n) - \log \log B}{\sqrt{\log \log B}}$$
converge in distribution to the standard normal distribution (throughout the paper all finite sets are equipped with the uniform probability measure).

In this paper we prove versions of this result for bad reduction types in families of curves. For applications, one often wants to detect finer arithmetic information than just bad reduction, such as whether the reduction is semistable.  We show that the primes of bad semistable reduction obey an Erd\H{o}s--Kac type theorem. 

\begin{theorem*} \hfill
\begin{enumerate}
	\item Over the set of hyperelliptic curves of fixed genus, the (renormalised) number of primes of bad semistable reduction is normally distributed.
	\item Over the set of plane curves of fixed degree, the (renormalised) number of primes of bad semistable reduction is normally distributed.
\end{enumerate}
\end{theorem*}
Our methods, which come from the paper \cite{EKM}, are robust enough to allow for more general families of curves under suitable assumptions. See Sections \ref{sec:hyperelliptic} and \ref{sec:plane_curves} for precise statements and further details. An immediate application of our results is the following: for any given $N > 0$, one hundred percent of degree $d$ hyperelliptic (resp.\ plane) curves $C$ have at least $N$ primes $p$ of bad semistable reduction. The case $N = 1$ for hyperelliptic curves is due to Van Bommel \cite{vB14}. In the text below, we prove a more precise  result, insisting that the reduction is semistable and irreducible over $\F_p$ with exactly one node, and hence that the Tamagawa number $c_p(\mathrm{Jac}(C))$ is equal to $1$.

\begin{acknowledgements}
Robert Lemke Oliver is supported by the National Science Foundation grant DMS-2200760 and by a Simons Fellowship in Mathematics.  Daniel Loughran is supported by  UKRI Future Leaders Fellowship MR/V021362/1. Ari Shnidman is funded by the European Union (ERC, CurveArithmetic, 101078157) and the Israel Science Foundation (grant No. 2301/20).
\end{acknowledgements}

\section{An Erd\H{o}s--Kac type result}\label{sec:EK}
We prove an Erd\H{o}s--Kac type theorem as an application of the main result of \cite{EKM}. We begin by recalling the set up of \cite[Thm.~1.8]{EKM}.

Let $X \subset \P^N_\Q$ be a quasi-projective variety over $\Q$. The naive height on projective space induces a height function $H: X(\Q) \to \R_{>0}$. Let $\mathcal{X}$ be a choice of model for $X$ over $\Z$.  The model allows us to define the set of integral points $\mathcal{X}(\Z)$.

We assume that $\mathcal{X}$ and the height $H$ satisfy the following properties. There exists a bound $A > 0$ and constants $M, \eta > 0$ such that for $Q \in \N$ with $\gcd(Q,\prod_{p \leq  A}p) = 1$ and for $\Upsilon \subset \mathcal{X}(\Z/Q\Z)$, we have 
\begin{equation} 
\label
{eqn:eff_equ}
\frac{\#\{ x \in \mathcal{X}(\Z): H(x) \leq B, x \bmod Q \in \Upsilon\}}
{\#\{ x \in \mathcal{X}(\Z): H(x) \leq B\}}
= \frac{\#\Upsilon}{\# \mathcal{X}(\Z/Q\Z)} + O(Q^{M} B^{-\eta})
\end{equation}
as $B \to \infty$. This condition is called \emph{effective equidistribution}, as it gives equidistribution in congruence classes with an explicit error term.

Let $\mathcal{Z}_1, \mathcal{Z}_2 \subset \mathcal{X}$ be closed subschemes. For $x \in \mathcal{X}(\Z)\setminus \mathcal{Z}_1(\Z)$, we define
\begin{equation} \label{def:omega_Z}
\omega_{\mathcal{Z}_1\setminus \mathcal{Z}_2}(x) = \#\{ p : x \bmod p \in \mathcal{Z}_1(\F_p) \setminus \mathcal{Z}_2(\F_p)\}.
\end{equation}
The condition $x \notin \mathcal{Z}_1(\Z)$ is easily seen to imply that the number of such primes is finite, hence this is well-defined. We also consider a multiplicity $1$ version of this. Namely for a Cartier divisor $\mathcal{D} \subset \mathcal{X}$ and closed subscheme $Z \subset \mathcal{X}$, for $x \in \mathcal{X}(\Z)\setminus \mathcal{D}(\Z)$ we define
\begin{equation} \label{def:omega_1}
\omega^1_{\mathcal{D} \setminus \mathcal{Z}}(x) = \#\{ p : x \text{ meets $\mathcal{D} \bmod p$ transversely
outside of } \mathcal{Z}\}.
\end{equation}
This condition means the following: Firstly that $x \bmod p \in  \mathcal{D}(\F_p) \setminus \mathcal{Z}(\F_p)$. Secondly, if $f = 0$ is a local equation for $\mathcal{D}$ around $x$, then $v_p(f(x))) = 1$.

\begin{theorem} 
\label
{thm:main}
	Let $X \subset \P^N_\Q$ be a normal quasi-projective variety  with induced height function $H$  and  $\mathcal{X}$ a choice of model for $X$ over $\Z$ which satisfy \eqref{eqn:eff_equ}. 
	Let $D_1,D_2 \subset X$ be reduced effective Cartier divisors with $D_1 \neq 0$
	and $D_1 \not \subseteq D_2$.
	Let $\mathcal{D}_i$ be their closures in $\mathcal{X}$. Let $c_{D_1 \setminus D_2}$ denote the number of 
	irreducible components of $D_1$ not contained in $D_2$. Then as $B \to \infty$, the random variables
	$$ \{x \in \mathcal{X}(\Z) \setminus \mathcal{D}_1(\Z) : H(x) \leq B\}  \to \R, \quad x \mapsto 
	\frac{\omega_{\mathcal{D}_1 \setminus \mathcal{D}_2}(x) - c_{D_1 \setminus D_2} \log \log B}{\sqrt{c_{D_1 \setminus D_2}\log \log B}}$$
	$$ \{x \in \mathcal{X}(\Z) \setminus \mathcal{D}_1(\Z) : H(x) \leq B\}  \to \R, \quad x \mapsto 
	\frac{\omega^1_{\mathcal{D}_1 \setminus \mathcal{D}_2}(x) - c_{D_1 \setminus D_2} \log \log B}{\sqrt{c_{D_1 \setminus D_2}\log \log B}}$$
	converge in distribution to a standard normal as $B \to \infty$.
\end{theorem}

We will prove this by applying the result \cite[Thm.~1.9]{EKM}. This concerns the function
	$$
	\omega_{\mathcal{D}_1}(x) = \#\{ p : x \bmod p \in \mathcal{D}_1(\F_p)\},
	$$
	and shows that 
	$$\frac{\omega_{\mathcal{D}_1}(x) - c_{D_1} \log \log B}{\sqrt{c_{D_1}\log \log B}}$$
	converges in distribution to a standard normal, where $c_{D_1}$ 
	denotes the number of irreducible components of $D_1$.  Thus, the new pieces in Theorem \ref{thm:main} amount simply to showing that the imposition of the further conditions defining $\omega_{\mathcal{D}_1\setminus \mathcal{D}_2}$ and $\omega^1_{\mathcal{D}_1\setminus \mathcal{D}_2}$ do not affect the limiting distribution compared to the divisor $D_1$.  We do so by establishing in the following pair of lemmas that the number of primes at which these definitions possibly differ have bounded moments.  Consequently, after accounting for the normalizing factor $\sqrt{\log\log B}$ that tends to infinity, these primes will have no impact on the limiting distribution.
	
\begin{lemma} \label{lem:codim_2}
	Let $\mathcal{Z} \subset \mathcal{X}$ be closed of  codimension $2$. Then for each integer $k \geq 1$,
	$$
		\limsup_{B \to \infty} \frac{1}{\#\{x \in \mathcal{X}(\Z) : H(x) \leq B\}} \sum_{\substack{x \in \mathcal{X}(\Z) \setminus \mathcal{Z}(\Z) \\ H(x) \leq B}} \omega_{\mathcal{Z}}(x)^k
	$$
	exists.
\end{lemma}
\begin{proof}
	We begin by considering for any $x \in \mathcal{X}(\Z) \setminus \mathcal{Z}(\Z)$ and any $y \geq 1$ the moments of the related function
		\[
			\omega_{\mathcal{Z},y}(x) := \#\{ p \leq y : x \bmod p \in \mathcal{Z}(\F_p) \}.
		\]
	Changing the order of summation, we find
		\[
			\sum_{\substack{x \in \mathcal{X}(\Z) \setminus \mathcal{Z}(\Z) \\ H(x) \leq B}} \omega_{\mathcal{Z},y}(x)^k
				= \sum_{p_1,\dots,p_k \leq y} \#\{ x \in \mathcal{X}(\Z)\setminus \mathcal{Z}(\Z) :
	H(x) \leq B, x \bmod p_i \in \mathcal{Z}(\F_{p_i}) \, \forall i \leq k\}.
		\]
	By the Lang--Weil estimates \cite{LW54} we have $\#\mathcal{X}(\F_p) \sim p^n$
	and $\#\mathcal{Z}(\F_p) \ll p^{n-2}$ where $n = \dim X$.
	We apply our equidistribution assumption \eqref{eqn:eff_equ} to thus obtain
		\[
			\frac{1}{\#\{x \in \mathcal{X}(\Z) : H(x) \leq B\}} \sum_{\substack{x \in \mathcal{X}(\Z) \setminus \mathcal{Z}(\Z) \\ H(x) \leq B}}
	\omega_{\mathcal{Z},y}(x)^k 
				\ll \sum_{p_1,\dots,p_k \leq y} \frac{1}{\mathrm{lcm}(p_1,\dots,p_k)^2}  + O(B^{-\eta} y^{kM}).
		\]
	The summation above converges as $y\to \infty$, and choosing $y=B^{\eta/kM}$, the error term remains bounded.  Thus,
		\begin{equation} \label{eqn:omega-y-limit}
			\limsup_{B\to\infty} \frac{1}{\#\{x \in \mathcal{X}(\Z) : H(x) \leq B\}} \sum_{\substack{x \in \mathcal{X}(\Z) \setminus \mathcal{Z}(\Z) \\ H(x) \leq B}}
	\omega_{\mathcal{Z},y}(x)^k 
		\end{equation}
	exists.  To compare $\omega_{\mathcal{Z}}$ with $\omega_{\mathcal{Z},y}$, we note that
	if $x \bmod p \in \mathcal{Z}(\F_p)$ then $p \ll H(x)^d$, where $d = \deg \mathcal{Z}$.	
	This gives
		\[
			\omega_{\mathcal{Z}}(x) - \omega_{\mathcal{Z},y}(x) 
				\leq 1 + \frac{\log B^d}{\log y} \ll \frac{dkM}{\eta}
		\]
	by our choice $y = B^{\eta/kM}$.
	This implies that
		\[
			\omega_{\mathcal{Z}}(x)^k
				= \omega_{\mathcal{Z},y}(x)^k + O\left(\frac{dkM}{\eta} \omega_{\mathcal{Z},y}(x)^{k-1}\right),
		\]
	and so 
		\[
			\limsup_{B \to \infty} \frac{1}{\#\Omega_B} \sum_{\substack{x \in \mathcal{X}(\Z) \setminus \mathcal{Z}(\Z) \\ H(x) \leq B}} \omega_{\mathcal{Z}}(x)^k
		\]
	must exist by comparison to the analogous quantity \eqref{eqn:omega-y-limit}.
\end{proof}

\begin{lemma} \label{lem:transverse}
	For each integer $k \geq 1$,
	$$
		\limsup_{B \to \infty} \frac{1}{\#\{x \in \mathcal{X}(\Z) : H(x) \leq B\}} 
		\sum_{\substack{x \in \mathcal{X}(\Z) \setminus \mathcal{D}_1(\Z) \\ H(x) \leq B}} (\omega_{\mathcal{D}_1 \setminus \mathcal{D}_2}(x) - \omega^1_{\mathcal{D}_1 \setminus \mathcal{D}_2}(x))^k
	$$
	exists.
\end{lemma}
\begin{proof}
	Let	$\mathcal{D}:=\mathcal{D}_1$ and let
	$\mathcal{Z} \subset \mathcal{X}$ denote the union of the non-smooth locus of $\mathcal{D}$
	and the restriction of the non-smooth locus of $\mathcal{X}$ to $\mathcal{D}$. 
	We have $\mathcal{Z} \neq \mathcal{D}$ as $X$ is normal and $D_1$ is reduced.
	Then
	$$\omega_{\mathcal{D}_1 \setminus \mathcal{D}_2}(x) - \omega^1_{\mathcal{D}_1 \setminus \mathcal{D}_2}(x) \ll
	|\omega_{\mathcal{D}}(x) - \omega^1_{\mathcal{D}}(x)|
	\ll 
	|\omega_{\mathcal{D} \setminus \mathcal{Z}}(x) - \omega^1_{\mathcal{D} \setminus \mathcal{Z}}(x)|
	+ \omega_{\mathcal{Z}}(x). $$
	By Lemma \ref{lem:codim_2} the moments of $\omega_{\mathcal{Z}}$ exist, so it suffices to consider
	$\omega_{\mathcal{D} \setminus \mathcal{Z}}(x) - \omega^1_{\mathcal{D} \setminus \mathcal{Z}}(x)$.
	On the one hand, by \cite[Cor.~2.4]{BL19} we have
	$$\#\{ x \in \mathcal{X}(\Z/p^2\Z) : x \text{ meets $\mathcal{D} \bmod p$ 
	non-transversely in a smooth point} \}	\ll p^{2n - 2},$$
	where $n = \dim X$. On the other hand, by the Lang--Weil estimates \cite{LW54} and Hensel's lemma 
	\cite[Lem.~2.1]{BL19} applied to the smooth locus of $\mathcal{X}$, 
	we have $\#\mathcal{X}(\Z/p^2\Z) \gg p^{2n}$. 
	It follows that the proportion of $x \in \mathcal{X}(\Z/p^2\Z)$
	which meet $\mathcal{D}$ non-transversely in a smooth point is $O(1/p^2)$. We now proceed exactly as in the proof of Lemma \ref{lem:codim_2}.
\end{proof}

	We now complete the proof of Theorem \ref{thm:main}.

\begin{proof}[Proof of Theorem \ref{thm:main}]
	Applying \cite[Thm.~1.9]{EKM} to the function
	$$
	\omega_{\mathcal{D}_1}(x) = \#\{ p : x \bmod p \in \mathcal{D}_1(\F_p)\}
	$$
	shows that 
	$$\frac{\omega_{\mathcal{D}_1}(x) - c_{D_1} \log \log B}{\sqrt{c_{D_1}\log \log B}}$$
	converges in distribution to a standard normal, where $c_{D_1}$ 
	denotes the number of irreducible components of $D_1$. 
	Write $D_1 \cap D_2 = E \sqcup Z$ where $E$ is a divisor and $Z$ has codimension $2$ in $X$.
	Let $\mathcal{E}$ and $\mathcal{Z}$ be their respective closures in $\mathcal{X}$. 
	As $E$ and $Z$ are disjoint, we have
	$$\omega_{\mathcal{D}_1 \setminus \mathcal{D}_2}(x) = 
	\omega_{\mathcal{D}_1}(x) - \omega_{\mathcal{E}}(x) - \omega_{\mathcal{Z}}(x) + O(1).$$
	Using $c_{D_1} = c_{D_1 \setminus D_2} + c_{E}$, to prove the first part,
	it thus suffices to show that
	$$\frac{\omega_{\mathcal{Z}}(x)}{\sqrt{\log \log B}}$$
	converges in distribution to $0$. 
	However by Lemma \ref{lem:codim_2} we have
	$$\lim_{B \to \infty} \frac{\omega_{\mathcal{Z}}(x)^k}{(\log \log B)^{k/2}}
	= 0 $$
	for every integer $k \geq 1$, which shows the desired claim.
	For the second part, it suffices to show that 
	$$\frac{\omega_{\mathcal{D}_1 \setminus \mathcal{D}_2}(x) - \omega^1_{\mathcal{D}_1 \setminus \mathcal{D}_2}(x)}{\sqrt{\log \log B}}$$
	converges in distribution to $0$. This similarly follows from Lemma \ref{lem:transverse}.
\end{proof}

\begin{remark}A version of Theorem \ref{thm:main} will hold for general variants $\omega_{\mathcal{D}_1}^*(x)$ of $\omega_{\mathcal{D}_1}(x)$, like $\omega_{\mathcal{D}_1 \setminus \mathcal{D}_2}$ and $\omega^1_{\mathcal{D}_1 \setminus \mathcal{D}_2}$, provided the following holds: whether a prime $p$ counted by $\omega_{\mathcal{D}_1}(x)$ is not counted by $\omega_{\mathcal{D}_1}^*(x)$ is determined by congruence conditions $A \subset \mathcal{X}(\Z/p^k\Z)$ such that $(\#\mathcal{D}(\Z/p^k\Z) - \#A)/\#\mathcal{X}(\Z/p^k\Z) = O(p^{-1-\delta})$ for some divisor $\mathcal{D}\subset \mathcal{X}$ and some $\delta>0$.  
\end{remark}

\section{Application to polynomials}\label{sec:apps}

The rest of our results are based on the following simple application of Theorem \ref{thm:main}.
\begin{theorem} \label{thm:polynomials}
Let $h_1, h_2 \in \Z[x_1,\dots,x_n]$ be non-zero polynomials with $h_1$ non-constant and squarefree in $\Q[x_1,\dots,x_n]$. Let $c$  be the number of non-associated irreducible factors of $h_1$ not dividing $h_2$, and suppose that $c > 0$. Then the random variables
\begin{align*}
\{ \x \in \Z^n : h_1(\x) \neq 0, \| \x \| \leq B\} \to \R,& \quad \x \mapsto 
\frac{\#\{ p \mid h_1(\x) : p \nmid h_2(\x)\} - c \log \log B}{\sqrt{c \log \log B}} \\
\{ \x \in \Z^n : h_1(\x) \neq 0, \| \x \| \leq B\} \to \R,& \quad \x \mapsto 
\frac{\#\{ p  : v_p(h_1(\x)) =1,  p \nmid h_2(\x)\} - c \log \log B}{\sqrt{c \log \log B}} 
\end{align*}
converge in distribution to a standard normal. 
\end{theorem}
\begin{proof}
	Apply Theorem \ref{thm:main} with $\mathcal{X}= \A^n_\Z$ and $D_1 : h_1(\x) = 0$ and $D_2 : h_2(\x) = 0$.
\end{proof}

\begin{remark}
	Theorem \ref{thm:main} also gives versions of Theorem \ref{thm:polynomials} for 
	projective space instead of affine space. The corresponding effective equidistribution property
	is proven in \cite[Prop.~2.1]{LS21}.
\end{remark}

\begin{corollary}\label{cor:poly}
Let $h_1$ and $h_2$ be as in Theorem \ref{thm:polynomials}.  For $\t \in \Z^n$, let $\omega_{h_1,h_2}(\t)$ be the number of primes $p$ dividing $h_1(\t)$ but not $h_2(\t)$. Then
	$$\lim_{B \to \infty}
	\frac{\#\left\{ \t \in \Z^n : 
	\begin{array}{ll} h_1(\t) \neq 0, \| \t \| \leq B, \omega_{h_1,h_2}(\t) \geq (\log \log B)/(\log \log \log B)\\
	\end{array}
	\right\}}
	{\#\{ \t \in \Z^n :  \| \t \| \leq B\}} = 1.$$
\end{corollary}

Corollary \ref{cor:poly} was used in \cite{BFS} to show that $100\%$ of specializations in a certain family of genus two Jacobians have at least $N$ primes of semistable bad reduction (for any fixed $N$). In the rest of this paper, we show how to deduce similar results about rather general families of curves and abelian varieties. 

\section{Semistable reduction}
We recall some basic properties of semistable curves and Jacobians \cite[Tag 0E6X]{Stacks}. Another good reference is \cite[\S 8-10]{Liu}, but the definitions there are slightly different.  

Let $C$ be a geometrically connected projective curve over a field $F$, and assume the genus $g = \dim_F H^1(C, \mathcal{O}_C)$ is at least $1$.  Let $\bar F$ be an algebraic closure of $F$ and $C_{\bar F}$ the base change of $C$ to $\bar F$. Then $C$ is {\it smooth} if $C_{\bar F}$ is smooth over $\bar F$ (and in particular, irreducible).  The curve $C$ is {\it semistable} if $C_{\bar F}$ is smooth over $\bar F$ apart from finitely many nodes, and has no irreducible components isomorphic to $\P^1_{\bar F}$ that meet the rest of $C_{\bar F}$ in only one point. This last condition excludes curves like $\P^1$ or $\{xy = 0\} \subset \P^2$, consistent with the condition $g \geq 1$. 

\begin{definition} A smooth proper geometrically integral curve $C$ of genus $g \geq 1$ over $\Q$ has {\it good $($resp.\ semistable$)$ reduction} at $p$ if there exists a proper model $\mathcal{C}$ of $C_{\Q_p}$ over $\Spec \, \Z_p$ such that the special fibre $\mathcal{C}_{\F_p}$ is a smooth (resp.\ semistable) curve over $\F_p$. 
\end{definition} 

We say $C$ has {\it bad reduction} at $p$ if it does not have good reduction at $p$. 

\begin{remark}\label{rem:stable terminology}
$C/\Q_p$ is semistable if and only if its minimal proper regular model $\mathcal{C}/\Spec\, \Z_p$ has semistable special fiber \cite[10.3.34]{Liu}.  Moreover, if $F/\Q_p$ is a finite extension with ring of integers $R \subset F$, then the base change $\mathcal{C}_{\Spec\, R}$ is the minimal proper regular model for $C_F$ \cite[10.3.36]{Liu}. Thus if $C/\Q_p$ admits at least one bad but semistable model, then all other models are bad as well and $C$ has bad reduction over every finite extension of $\Q_p$. In other words, the good/bad reduction type of a semistable curve is `stable'. 
\end{remark}

Now let $A/\Q_p$ be an abelian variety, and let $\mathcal{A}$ be its N\'eron model over $\Spec \, \Z_p$ with special fiber $\mathcal{A}_{\F_p}$. The connected component of the identity $\mathcal{A}_{\F_p}^0$ is a geometrically connected commutative algebraic group over $\F_p$, hence sits in a short exact sequence
\[0 \to U \times T \to \mathcal{A}_{\F_p}^0 \to B \to 0,\]
where $B$ is an abelian variety, $T$ is a torus, and $U$ is a unipotent group. The numbers $a = \dim B$, $t = \dim T$, and $u = \dim U$ are the {\it abelian}, {\it toric}, and {\it unipotent ranks} of $A$ respectively. We have $u + t + a = \dim A$.
     
\begin{definition}
An abelian variety $A$ over $\Q$ has {\it good} (resp.\ {\it semistable}) {\it reduction at} a prime $p$ if the connected component of the identity of the special fibre $\mathcal{A}_{\F_p}$ of its N\'eron model $\mathcal{A}$ over $\Z_p$ is an abelian variety (resp.\ semi-abelian variety). 
\end{definition}

Thus $A/\Q_p$ has good (resp.\ semistable) reduction at $p$ if and only if $u + t = 0$ (resp.\ $u= 0$). In the definition above, we can equivalently ask for the existence of {\it some} proper model of $A_{\Q_p}$ over $\Z_p$ with the corresponding property in the special fibre.

If $C/F$ is a smooth curve, we write $\mathrm{Jac}(C) = \Pic^0(C)$ for its Jacobian, the abelian variety over $F$ of dimension $g$ parameterizing degree zero line bundles on $C$. 
\begin{lemma}\label{lem:DM}
Let $C/\Q$ be a smooth proper geometrically integral curve of genus $g \geq 1$. 
\begin{enumerate}[$(1)$]
\item  $C$ {\it has semistable reduction}  at $p$ if and only if $\mathrm{Jac}(C)$ has semistable reduction at $p$. 
\item If $\mathcal{C}$ is a semistable model for $C$ over $\Z_p$, then the toric rank of $\mathrm{Jac}(C)$ is equal to $m - c + 1$, where $m$ is the number of nodes in $\mathcal{C}_{\F_p}$ and $c$ is the number of irreducible components.  
\end{enumerate}
\end{lemma}
\begin{proof}
$(1)$ is a special case of \cite[Thm.\ 2.4]{DM} and $(2)$ is \cite[7.5.18]{Liu}.
\end{proof}

If a curve or abelian variety over $\Q$ has semistable but not good reduction at $p$, then we say it has {\it bad semistable reduction at} $p$.
\begin{remark}
If $C$ has good reduction, then so does $J$, by Lemma \ref{lem:DM}(2).  The converse may fail, however. For example, say $g = 2$ and $C$ reduces to a union $E \cup E'$ of two elliptic curves over $\F_p$ intersecting at a node. Then $C$ has bad semistable reduction, but $J$ has toric rank $0$ by  Lemma \ref{lem:DM}, and hence good reduction.  In fact, $J$ reduces to $E \times E'$ in this case. 
\end{remark}

If $A/\Q_p$ is an abelian variety with N\'eron model $\mathcal{A}$, the Tamagawa number $c_p(A)$ is by definition the number of $\F_p$-rational components of the group $\mathcal{A}_{\F_p}$. This is a crude measure of {\it how bad} the reduction of $A$ is at $p$. For instance, if $A$ has good reduction then $c_p(A) = 1$. The converse, however, is not true, as the following example shows.

\begin{example}\label{ex:ellipticcurve}
If $E/\Q$ is an elliptic curve with squarefree discriminant $\Delta$, then by Tate's algorithm \cite{Tate}, we have $c_p(E) = 1$ for all primes $p$, including those dividing $\Delta$. 
\end{example}   

More generally, the Tamagawa number $c_p(J)$ of a semistable Jacobian $J = \mathrm{Jac}(C)$ can be computed from the intersection matrix of the irreducible components of the special fiber of a minimal proper regular model $\mathcal{C}$ over $\Z_p$. This uses Raynaud's theorem, that the N\'eron model of $J$ is represented by $\Pic^0_{\mathcal{C}/\Z_p}$.  See \cite[\S 9.6]{BLR} for more details. 

In the next two sections, we consider families of curves $\{C_\t\}_{\t \in \Z^n}$ and prove Erd\H{o}s--Kac type laws for the number of primes of bad semistable reduction for specializations $\t$ of bounded height. 
For the sake of applications, we will prove a more precise result, which shows that it is primes of {\it minimally} bad  reduction that play the role of prime numbers under this analogy. The notion of `minimally bad reduction' will generalize Example~\ref{ex:ellipticcurve}: a certain discriminant polynomial will have $p$-adic valuation 1, which will imply that $C_\t$ and $\Jac(C_\t)$ have bad semistable reduction and moreover $c_p(\Jac(C_\t)) = 1$.      

\section{Families of hyperelliptic curves}\label{sec:hyperelliptic}

Let $g \geq 1$, $n \geq 1$ and let $a_0,\ldots,a_{2g+2} \in \Q[t_1,\ldots,t_n]$ be polynomials with integer coefficients. We consider the corresponding family 
$$y^2 = f_\t(x) := a_{2g+2}(\t)x^{2g+2} + \dots + a_1(\t)x + a_0(\t)$$
of hyperelliptic curves over $\A^n$. Denote by $\Delta(\t) \in \Q[t_1,\dots,t_n]$ the discriminant of $f_\t(x)$.
We say $C_\t$ has {\it minimally bad reduction} at $p$ if $v_p(\Delta(\t)) = 1$. 

\begin{lemma}\label{lem:disc1}
If $v_p(\Delta(\t)) = 1$, then
\begin{enumerate}[$(a)$]
\item both $C_\t$ and $\Jac(C_\t)$ have bad semistable reduction;
\item the curve $y^2 = f_\t(x,z)$ over $\Z_p$ is a minimal proper regular model for $C_\t$;
\item $c_p(\Jac(C_\t)) = 1$.
\end{enumerate}  
\end{lemma}
\begin{proof}
If $f_\t$ has a root of multiplicity three or higher over $\F_p$, then by Dedekind's theorem,  we have $v_p(\Delta(\t)) \geq 2$. It follows that $f_\t$ has at most double roots over $\F_p$, and since the valuation of $\Delta(\t)$ is 1 it must have exactly one double root. For such hyperelliptic curves, the homogenization $\mathcal{C}_\t \colon y^2 = f_\t(x,z)$ is a minimal regular model for $C_\t$ over $\Z_p$, and the special fiber is an irreducible genus $g-1$ curve with a simple node \cite[8.3.53]{Liu}. Hence $C_\t$ has bad semistable reduction, and moreover the component group of $\Jac(C_\t)$ is trivial (\cite[\S 9.6]{BLR}). By Lemma \ref{lem:DM}, the toric rank of $\mathrm{Jac}(C_\t)$ is $1$, hence  $\mathrm{Jac}(C_\t)$ has bad reduction as well. 
\end{proof}
We say $h \in \Q[x_1,\ldots, x_n]$ is {\it squarefull} if each irreducible factor $g \mid h$ satisfies $g^2 \mid h$.     
\begin{theorem}\label{thm:hyperfamilies}
Assume $\Delta(\t)$ is non-constant and not squarefull. Let $c$ be the number of non-associated irreducible polynomials $h(\t)$ exactly dividing $\Delta(\t)$. Then the random variables
\begin{align*}
\{ \t \in \Z^n & : \Delta(\t) \neq 0, \| \t \| \leq B\} \to \R,  \\
 \t &\mapsto 
\frac{\#\{ p  : C_\t \text{ has minimally bad reduction at } p \} - c \log \log B}{\sqrt{c \log \log B}}
\end{align*}
converge in distribution to a standard normal as $B \to \infty$.
\end{theorem}
\begin{proof}
Write $\Delta(\t) = \prod_{i = 1}^c h_i(\t) \prod_{i = c+1}^k h_i(\t)^{a_i}$, where the $h_i$ are irreducible and pairwise coprime, and $a_i \geq 2$ for $i > c$. We apply Theorem \ref{thm:polynomials} with $f_1 = \prod_{i = 1}^ch_i(\t)$ and $f_2 = \Delta(\t)/f_1$. By construction, the second counting function in Theorem \ref{thm:polynomials} exactly counts the number of primes $p$ of minimally bad (and by Lemma \ref{lem:disc1}, semistable) reduction for $C_\t$.  
\end{proof}

As an example, we apply this to the family of all hyperelliptic curves. 
\begin{corollary} \label{cor:standard}
Consider the family 
$$y^2 = a_{2g+2}x^{2g+2} + \dots + a_1x + a_0$$
of all hyperelliptic curves over $\A^{2g+2}$.
As $B \to \infty$, the random variables
\begin{align*}
\{ \mathbf{a} \in \Z^{2g+2} & : \Delta(\a) \neq 0, \| \mathbf{a} \| \leq B\} \to \R, \\
 \mathbf{a} &\mapsto 
\frac{\#\{ p  : C_{\mathbf{a}} \text{ has minimally bad reduction at } p\} - \log \log B}{\sqrt{ \log \log B}}\\
\{ \mathbf{a} \in \Z^{2g+2} & : \Delta(\a) \neq 0, \| \mathbf{a} \| \leq B\} \to \R, \\
 \mathbf{a} &\mapsto 
\frac{\#\{ p  : C_{\mathbf{a}} \text{ has bad reduction at } p\} - \log \log B}{\sqrt{ \log \log B}}
\end{align*}
converge in distribution to a standard normal.
\end{corollary}
\begin{proof}
The polynomial $\Delta$ is irreducible as an element of $\mathbb{C}[a_0,\dots, a_{2g+1}]$ (see
\cite[Ex.\ 1.4]{GKZ}). Applying Theorem \ref{thm:polynomials} to $\Delta$ we see that only primes of minimally bad reduction contribute to the distributions.  
\end{proof}

Corollary \ref{cor:standard} implies, as in Corollary \ref{cor:poly}, that for $100\%$ of hyperelliptic curves, both $C_\t$ and $\Jac(C_\t)$ have
bad semistable reduction for at least $N$ primes, for any $N > 0$. The case $N = 1$ of this result for a related family of hyperelliptic curves is due to Van Bommel \cite{vB14}.

For families which may not satisfy the hypotheses of Theorem \ref{thm:hyperfamilies}, we prove the following variant whose conclusion is a bit weaker.   Denote by $\Delta'(\t)$ the discriminant of $f'_\t(x):= \frac{\mathrm{d} f_\t(x)}{ \mathrm{d} x}.$

\begin{theorem} \label{thm:hyperelliptic_weak}
Assume that $\Delta(\t)$ and $\Delta'(\t)$ are non-constant and  let $c$ be the number of non-associated irreducible factors of $\Delta(\t)$ not dividing $\Delta'(\t)$. 
\begin{enumerate}
\item If $(\Delta'(\t)) \not \subseteq \rad(\Delta(\t))$, then the random variables
\begin{align*}
\{ \t \in \Z^n & : \Delta(\t) \neq 0, \| \t \| \leq B\} \to \R,  \\
 \t &\mapsto 
\frac{\#\{ p  : \Jac(C_\t) \text{ has bad semistable reduction at } p \text{ and } p \nmid \Delta'(\t)\} - c \log \log B}{\sqrt{c \log \log B}}
\end{align*}
converge in distribution to a standard normal as $B \to \infty$. 
\item If $\Delta(\t)$ and $\Delta'(\t)$ are coprime then the random variables
\begin{align*}
\{ \t \in \Z^n & : \Delta(\t) \neq 0, \| \t \| \leq B\} \to \R,  \\
 \t &\mapsto 
\frac{\#\{ p  : \Jac(C_\t) \text{ has bad semistable reduction at } p \} - c \log \log B}{\sqrt{c \log \log B}}
\end{align*}
converge in distribution to a standard normal as $B \to \infty$. 
\end{enumerate}
\end{theorem}

\begin{proof}
By Lemma \ref{lem:disc}, this follows from Theorem \ref{thm:polynomials} with $h_2 = \Delta'$ and $h_1$ a generator of the ideal $\rad(\Delta)$. 
\end{proof}

\begin{lemma} \label{lem:disc}
Let $\t \in \Z^n$ and suppose $p > 2g+2$. If $p \mid \Delta(\t)$ but $p \nmid \Delta'(\t)$,  then both $C_{\t}$ and $\mathrm{Jac}(C_\t)$ have bad semistable reduction.
\end{lemma}
\begin{proof}
This is well-known, but we sketch a proof for completeness. We have $p \mid \Delta(\t)$ if and only if $f_\t(x)$ has a root of multiplicity at least two over $\F_p$. Since $p \nmid \Delta'(\t)=\disc(f'_\t)$, all such roots must have multiplicity equal to two. 
Since $\Delta'(\t) = \mathrm{Res}_x(f'_\t, f''_\t)$, the condition $p \nmid \Delta'(\t)$ also implies that $p$ does not divide all the coefficients of $f$.  We may assume the curve $C_{\t} \colon y^2 = f_\t(x)$ over $\F_p$ has no points at infinity. (If $p$ happens to divide the leading coefficient, we change coordinates so that $\infty \in \P^1(\F_p)$ is not a root of the homogenization of $f_\t$ over $\F_p$; the condition $p > 2g+2$ guarantees that this is always possible.) This curve is then smooth aside from singularities \'etale locally of the form $y^2 = x^2g(x)$, where $g$ is non-vanishing at $x=0$. The singularities are therefore  nodes, say $m > 0$ of them. 

If $m < g+1 = \frac12\deg(f_\t)$, then there is one (singular) irreducible component of genus $g - m$, so $C_{\t}$ has bad semistable reduction.  By Lemma \ref{lem:DM}, the toric rank is $m > 0$, so $J_{\t}$ also has bad semistable reduction. If $m = g+1$, then there are two (non-singular) irreducible components crossing at $m \geq 2$ points so again $C_{\it}$ is semistable, and the toric rank is  $g > 0$, so $J_\t$ has bad semistable reduction as well.
\end{proof}

\begin{remark}
See \cite[eq.\ (10)]{BFS} for examples where Theorem \ref{thm:hyperelliptic_weak} applies but Theorem \ref{thm:hyperfamilies} does not. 
\end{remark}

For families of Jacobians with everywhere potentially good reduction, the hypotheses of Theorems \ref{thm:hyperfamilies} and \ref{thm:hyperelliptic_weak} are evidently not satisfied: if $J$ has potentially good reduction at $p$, then it cannot have bad semistable reduction at $p$.

A simple example is the family $y^2 = x^\ell + t$. For every $t$, the Jacobian has potentially good reduction since it has complex multiplication over $\Q(\zeta_\ell)$, but $\Delta(t) = -\ell^\ell t^{\ell-1}$ and $\Delta'(t) = 0$, so both Theorems do not apply.  Similarly, Theorems \ref{thm:hyperfamilies} and \ref{thm:hyperelliptic_weak} do not apply in twist families such as $C_t \colon ty^2 = f(x)$, which necessarily have finitely many primes of bad semistable reduction in the entire family.       

\begin{example}
For a non-isotrivial example, consider the curves $C_t \colon y^2 = f_t(x)$, where
\[f_t= (x^2 + 2x - 2) (x^4 + 4x^3 + (2t^2 - 8)x - t^2 + 4).\]
We have $\Delta(t) = -2^63^6(t^2-4)^2t^{12}$ whereas $\Delta'(t) = -2^83^8(t^2 -4)t^4g(t)$ for some irreducible sextic polynomial $g(t)$, so Theorem \ref{thm:hyperelliptic_weak} does not apply.
The Jacobian of any curve in this family, which is taken from \cite{LS}, has quaternionic multiplication by the quaternion algebra of discriminant $6$, and hence has no primes of bad semistable reduction.  
\end{example}

Is there a geometric characterization of the families of hyperelliptic curves not satisfying the condition $(\Delta'(\t)) \not \subseteq \rad(\Delta(\t))$ of Theorem \ref{thm:hyperelliptic_weak}(1)? All examples that we encountered so far are either isotrivial or have Jacobians with large endomorphism algebra. In particular, they have everywhere potentially good reduction aside from finite many primes which depend only on the family.

\section{Families of plane curves} \label{sec:plane_curves}

Let $V_d$ be the space of homogeneous polynomials $f(x,y,z)$ of degree $d$.  There is a polynomial $\Delta = \Delta_d$ on $V_d$, called the {\it discriminant}, with the property that $\Delta(f) = 0$ if and only if the curve $C_f \colon f(x,y,z) = 0$ is singular \cite[13.1.D]{GKZ}. 

We say $C_f$ has {\it minimally bad reduction} at a prime $p$ if $v_p(\Delta) = 1$. The equation $f(x,y,z) = 0$ then gives a minimal regular model for $C_f$ over $\Z_p$, and the singular locus in the special fiber is a single node \cite[Thm.\ 1.1]{PS}. By Bezout's theorem, the special fiber is irreducible, and it follows that $C_f$ has bad semistable reduction at $p$. The toric rank of $\Jac(C_f)$ is $1$ by Lemma \ref{lem:DM}, so $J$ also has bad semistable reduction at $p$. By \cite[\S 9.6]{BLR} and the irreducibility of the special fiber, we have $c_p(J)= 1$ as well.

\begin{theorem}\label{thm:planecurves}
Let $d \geq 3$. Consider the family 
$$\sum a_{ijk} x^iy^jz^k = 0 $$
of all degree $d$ plane curves over affine $2 +d \choose d$-space.
Then as $B \to \infty$, the random variables
\begin{align*}
\{ \mathbf{a} \in \Z^{2 +d \choose d} & : \Delta(\a) \neq 0, \| \mathbf{a} \| \leq B\} \to \R, \\
 \mathbf{a} &\mapsto 
\frac{\#\{ p  : C_{\mathbf{a}} \text{ has minimally bad reduction at } p\} - \log \log B}{\sqrt{ \log \log B}}
\end{align*}
converge in distribution to a standard normal.
\end{theorem}

\begin{proof}
This follows from Theorem \ref{thm:polynomials} and \cite[Thm.\ 1.1]{PS}, and the fact that $\Delta$ is an irreducible polynomial in the $a_{ijk}$ \cite[\S 13.1.D]{GKZ}.
\end{proof}
 Just as in Theorem \ref{thm:hyperfamilies}, Theorem \ref{thm:planecurves} generalizes immediately to parameterized families of plane curves $C_{\t}$ over $\A^n$ such that $\Delta(\t)$ is not squarefull.
 
 For more general families, we prove an analogue of Theorem \ref{thm:hyperelliptic_weak}. For $f = \sum a_{ijk} x^iy^jz^k$, let $H_{xy} = f_{xx}f_{yy} - f_{xy}^2$ be the upper left 2-by-2 minor of its Hessian matrix. The resultant $R(f) = \mathrm{Res}(H_{xy}, f_x,f_y)$ is a polynomial in the $a_{ijk}$ which vanishes precisely when $H_{xy},f_x,$ and $f_y$ have a common zero \cite[\S 13]{GKZ}.

\begin{lemma}
$H_{xy}$ vanishes whenever $C_f$ has a non-nodal singularity. Hence so does $R(f)$. 
\end{lemma}
\begin{proof}
Let $H = H(f)$ be the 3-by-3 Hessian matrix of double partial derivatives.  
Suppose $C_f$ is singular at a point $P$. Then $P$ is a triple point (or worse) if and only if the matrix $H(P)$ vanishes identically, i.e. has rank 0.  If $P$ is a double point, then the rank of $H(P)$ is either one or two, and in the latter case $P$ is an ordinary double point (i.e.\ a node) since the tangent lines are separated.  Thus $H_{xy}$ vanishes at all singular points which are not nodes.  
\end{proof}
\begin{proposition}\label{prop:bad}
Let $f \in V_d(\Z_p)$, $\Delta(f) \neq 0$. Suppose $p \mid \Delta(f)$ but $p \nmid R(f)$. Then both $C_f$ and $\Jac(C_f)$ have bad semistable reduction at $p$.  
\end{proposition}

\begin{proof}
By assumption $C_f$ is proper over $\Z_p$ with only nodal singularities in the special fibre. It therefore has a semistable model over $\Z_p$ \cite[\href{https://stacks.math.columbia.edu/tag/0CDG}{Lemma 0CDG}]{Stacks}. In fact, we claim that $C_f$ is itself semistable over $\Z_p$. Assume otherwise. Write $\bar{f} = \prod f_i$, with $f_i\in \F_p[x,y,z]$ irreducible.  The condition $p \nmid R(f)$ implies that $f_i \neq f_j$ for $i \neq j$, in other words the reduction $C_{f,p}/\F_p$ is reduced, with irreducible components $C_i = \{f_i = 0\}$, for $i = 1,\ldots, r$.  
We may assume $r > 1$. (If $r = 1$, then $C_{f,p} = C_1$ has only nodal singularities and hence is semistable.) As $C_f$ is not semistable, at least one of the $C_i$ has genus 0 intersecting the rest of the special fibre in a single reduced point. Since $d \geq 3$, this cannot happen by Bezout's theorem.   

By Remark \ref{rem:stable terminology}, 
 $C$ has bad semistable reduction.
To prove that $J$ also has {\it bad} semistable reduction, we need to show that the toric rank of $J$ is non-zero, or equivalently, that the abelian rank is strictly less than $g = (d-1)(d-2)/2$.  However, from the above semistable model we see that the abelian rank is at most $g' = \frac12\sum_{i = 1}^r (d_i - 1)(d_i-2)$, where $d_i = \deg(f_i)$ and $\sum d_i = d$.  Since $g' < g$ if $r \neq 1$, we may assume that $C_{f,p}$ is irreducible of degree $d$ and with $t \geq 1$ nodes.  But then the abelian rank of $J$ is $\frac12(d-1)(d-2) - t < g$, as claimed.
\end{proof}

For simplicity, we state only the analogue of Theorem \ref{thm:hyperelliptic_weak}.(2) in this setting.

\begin{theorem}\label{thm:planecurves-weak}
Let $a_{ijk}(\t) \in \Z[t_1,\ldots,t_n]$, and consider the family $C_\t \colon f_\t(x,y,z) = 0$ of degree $d$ plane curves over $\Q$, where $f_\t =  \sum_{ijk} a_{ijk}(\t)x^iy^jz^k$. Assume that $\Delta(f_\t)$ and $R(f_\t)$ are non-constant and coprime. Let $c$ be the number of non-associated irreducible factors of $\Delta(\t)$.  Then as $B \to \infty$, 
the random variables
\begin{align*}
\{ \mathbf{t} \in \Z^{2 +d \choose d} & : \Delta(f_\t) \neq 0, \| \mathbf{t} \| \leq B\} \to \R, \\
 \t &\mapsto 
\frac{\#\{ p  : C_{\t} \text{ has bad semistable reduction at } p\} - c\log \log B}{\sqrt{ c\log \log B}}
\end{align*}
converge in distribution to a standard normal.
\end{theorem}

\begin{proof}
By Proposition \ref{prop:bad} it is enough to apply Theorem \ref{thm:polynomials}, using  $h_2 = R(f_\t)$ and $h_1$ a generator of the ideal $\rad(\Delta(f_\t))$. 
\end{proof}

\begin{remark}
Over $V_d$, the polynomials $\Delta$ and $R$ have no common factors: Since $\Delta$ is irreducible it is enough to exhibit a single $f \in V_d$ with $\Delta(f) = 0$ and $R(f) \neq 0$.  Consider the curve $C_f \colon x^d + y^d =xyz^{d-2}$, which has a node at the origin.  The resultant 
\[R(f) = \mathrm{Res}(d^2(d-1)^2x^{d-2}y^{d-2} - z^{2(d-2)}, dx^{d-1} - yz^{d-2}, dy^{d-1} - xz^{d-2}),\]
is non-vanishing, since the scheme cut out by these three polynomials is empty.  
\end{remark}

The results of this section generalize immediately to Erd\H{o}s--Kac type results for reduction types of degree $d$ hypersurfaces $H \colon h(\x) = 0$ in $\P^n$. Indeed, there is a discriminant polynomial $\Delta$ for such hypersurfaces \cite[13.1.D]{GKZ}. Moreover, $v_p(\Delta(h)) = 1$ implies that $H \otimes \Z_p$ is regular with a unique singular point (a node) in its special fiber \cite[Thm.\ 1.1]{PS}, hence $H$ has semistable reduction over $\Q_p$ in that case.


\begin{thebibliography}{99}
\bibitem{BL19}
T. D. Browning and D. Loughran, Sieving rational points on varieties. \textit{Trans. Amer. Math. Soc.} 371 (2019), no. 8, 5757--5785.

\bibitem{BFS} N. Bruin, V. Flynn, and A.\ Shnidman,
Genus two curves with full $\sqrt{3}$-level structure and Tate-Shafarevich groups, {\it Sel.\ Math.\ New Ser.} {\bf 29}, 42 (2023).

\bibitem{BLR}
S. Bosch, W. Lütkebohmert, and M. Raynaud, 
{\it Néron models}, Ergebnisse der
Mathematik und ihrer Grenzgebiete {\bf (3)} 21, Springer, 1990.

\bibitem{DM} P.\ Deligne and D.\ Mumford, The irreducibility of the space of curves of given genus, {\it Inst. Hautes Études Sci.\ Publ.\ Math.}  {\bf 36} (1969), 75–109.
      
\bibitem{EKM} D. El-Baz, D. Loughran, and E. Sofos,
Multivariate normal distribution for integral points on varieties,
Trans. Amer. Math. Soc.,  \textbf{375} (2022), no. 5, 3089–3128.

 \bibitem{EK40}
P. Erd\H{o}s and M. Kac,  The Gaussian law of errors in the theory of additive number theoretic functions. \emph{Amer. J. Math.}, \textbf{62}, (1940), 738--742.

\bibitem{GKZ} I. M. Gelfand, M. M. Kapranov, and A. V. Zelevinsky, Discriminants, resultants and multidimensional determinants. Reprint of the 1994 edition. Modern Birkhäuser Classics. Birkhäuser Boston, Inc., Boston, MA, 2008.

\bibitem{LW54}
S. Lang and  A. Weil, Number of points of varieties in finite fields. 
\textit{Amer. J. Math.} \textbf{76} (1954), 819--827.

\bibitem{LS} J. Laga and  A. Shnidman,
The geometry and arithmetic of bielliptic Picard curves, preprint (2023).

\bibitem{Liu}
Q. Liu, Algebraic geometry and arithmetic curves.
Oxford Graduate Texts in Mathematics, {\bf 6}. Oxford Science Publications.
Oxford University Press, Oxford,  2002. 

\bibitem{LS21}
D. Loughran and  E. Sofos, An Erdős-Kac law for local solubility in families of varieties. \textit{Selecta Math.} 	{\bf27} (2021), no. 3, Paper No. 42.


\bibitem{PS} B. Poonen and M. Stoll, 
The valuation of the discriminant of a hypersurface (2020), available at \textit{https://math.mit.edu/~poonen/papers/discriminant.pdf}.

\bibitem{Stacks} The Stacks project authors, The Stacks project, \url{https://stacks.math.columbia.edu}, (2021).

\bibitem{Tate}
Tate, J.  Algorithm for determining the type of a singular fiber in an elliptic pencil.
 Modular functions of one variable, IV (Proc. Internat. Summer School, Univ. Antwerp, Antwerp, 1972), 
 pp. 33--52. {\it Lecture Notes in Math.}, Vol. 476, Springer, Berlin,  1975. 

\bibitem{vB14}
R. van Bommel, Almost all hyperelliptic Jacobians have a bad semi-abelian prime,
Master's thesis, University of Leiden, 2014. \texttt{https://www.raymondvanbommel.nl/Master.pdf}.







\end{thebibliography}
\end{document}